\documentclass[11pt,reqno]{article}
\usepackage{amssymb, amsmath, amsthm, amsfonts, amscd, epsfig, subfig, gensymb}
\usepackage[dvipsnames]{xcolor}
\usepackage[utf8]{inputenc}
\usepackage[english]{babel}
\usepackage{bm}
\usepackage{bbm}
\usepackage{ulem}
\usepackage{graphicx, epsfig}
\usepackage{geometry,graphicx,pgfplots}

\usepackage[export]{adjustbox}
\usepackage[titletoc,toc,title]{appendix}

\usepackage{algorithm}

\usepackage{mathtools}

\usepackage{afterpage}
\usepackage{comment}

\newtheorem{theorem}{Theorem}[section]

\newtheorem{lemma}[theorem]{Lemma}

\newtheorem{remark}{Remark}

\def\Z{\mathbb{Z}}

\def\R{\mathbb{R}}

\newcommand{\Om}{\Omega}

\newcommand{\RR}{\mathbb{R}}

\newcommand{\eps}{{\varepsilon}}

\newcommand{\p}{\partial}

\newcommand{\beq}{\begin{equation}}
\newcommand{\eeq}{\end{equation}}
\newcommand{\RN}[1]{%
  \textup{\uppercase\expandafter{\romannumeral#1}}%
}

\numberwithin{equation}{section}
\numberwithin{figure}{section}

\begin{document}
\title{Bayesian optimization approach for tracking the location and orientation of a moving target using far-field data
	\thanks{\footnotesize
This work is supported by the National Research Foundation of Korea(NRF) grant funded by the Korea government(MSIT) (RS-2023-00242528, RS-2024-00359109).
}}

\date{}

\author{
Woojoo Lee\thanks{Department of Mathematical Sciences, Korea Advanced Institute of Science and Technology, 291 Daehak-ro, Yuseong-gu, Daejeon 34141, Republic of Korea (woojoo.lee@kaist.ac.kr, mklim@kaist.ac.kr).}\and
Mikyoung Lim\footnotemark[2]\and
Sangwoo Kang\thanks{Pusan National University, 2, Busandaehak-ro 63beon-gil, Geumjeong-gu, Busan 46241, Republic of Korea (sangwoo.kang@pusan.ac.kr)}
}

\maketitle

\begin{abstract}
We investigate the inverse scattering problem for tracking the location and orientation of a moving scatterer using a single incident field. We solve the problem by adopting the optimization approach with the objective function defined by the discrepancy in far-field data. We rigorously derive formulas for the far-field data under translation and rotation of the target and prove that the objective function is locally Lipschitz with respect to the orientation angle at the true angle. By integrating these formulas with the Bayesian optimization approach, we reduce the cost of objective function evaluations. For the instance of an unknown target, machine learning via fully connected neural networks is applied to identify the shape of the target. Numerical simulations for randomly generated shapes and trajectories demonstrate the effectiveness of the proposed method.
\end{abstract}


\section{Introduction}

Inverse scattering problems, which primarily aim to identify the geometric and material features of targets, have been extensively studied in various fields such as biomedical imaging \cite{Abubakar:2002:IBD}, nondestructive testing \cite{Marklein:2006:ISI, Salucci:2016:RTN}, and remote sensing \cite{Woodhouse:2017:IMR}. These problems intrinsically involve nonlinearity and ill-posedness. To overcome these difficulties, researchers have developed a variety of algorithms, which can be classified as iterative methods, decomposition methods, and sampling methods (refer to the survey in \cite{Luke:2003:NRT} for more details). 
Additionally, recent studies have explored novel techniques under limited measurement configurations, which are commonly encountered in real-life applications \cite{Kang:2022:MSM, Kang:2022:FIS, Huang:2022:BAI}.

In the case of moving targets, the main concern in the inverse scattering problems is to estimate the trajectories of the targets, which we refer to as the {\it tracking}. The tracking problem of moving targets has gained significant attention in practical fields such as autonomous driving \cite{Leon:2021:RTT}, robotics \cite{Robin:2016:MRT}, and radar imaging \cite{Wang:2012:MRI, Haworth:2007:DTM}. A major challenge in the tracking is that the target's motion is unpredictable. Various techniques have been developed depending on the model configuration. For instance, Cheney and Borden applied the linearized imaging theory for scattered waves to design a tracking method for a moving point source in the scalar wave equation \cite{Cheney:2008:IMT}. 
For the same problem in $\mathbb{R}^3$, Wang et al. established an algorithm to track both the location of the source and the intensity of the signal by employing Bayesian inference during the tracking process \cite{Wang:2023:LMS}. In addition, in the context of the conductivity problem, Ammari et al. investigated the tracking from the expansion coefficients of multipole expansions \cite{Ammari:2013:TMT}.

\begin{figure}[h!]
	\centering
	\includegraphics[width=0.26\textwidth]{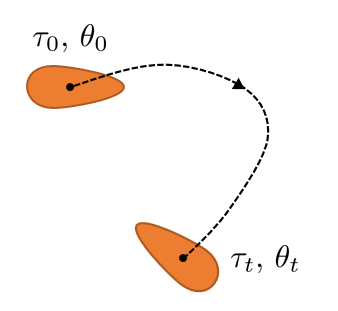}
	\caption{\label{fig:tracking} Example of tracking. We estimate the location $\tau_t$ and orientation $\theta_t$ at time $t$ from the far-field data at time $t$. The movement direction and shape orientation are considered independent.}
\end{figure}
  
In this paper, we address the tracking problem of identifying both the location and orientation of moving targets in two dimensions, where orientation refers to the rotation angle of the targets. We assume that the target scatterer satisfies the rigid motion over consecutive time intervals, where the shape of the target remains fixed, and its location and orientation change over time. We allow the movement direction and shape orientation to be independent of each other. This significantly increases the complexity of the tracking problem compared to cases where the movement direction and orientation coincide. Specifically,  we consider a sound-soft scatterer in two dimensions that moves in rigid motions. Our objective is to track both its location and orientation along its trajectory using the far-field data obtained from only a single incident field (see Figure \ref{fig:tracking}).

We adopt an optimization approach to retrieve the location and orientation angle of the target, where the objective function is defined by the discrepancy in two consecutive far-field data. We note that using only one incident field direction increases the difficulty in the inverse scattering problem. Moreover, the rotation angle lacks a one-to-one correspondence with the far-field data (see, for example, Remark \ref{rmk:nonuniqueness}). We address these challenges by investigating the properties of the far-field data with respect to rigid motions of the target. We derive the translation formula explicitly stating the relation between the shift of the target and the far-field data. We also rigorously show the local Lipschitz property of the objective function at the true angle.

To enhance the efficiency of our method, we incorporate Bayesian optimization, which is crucial due to the high computational cost of evaluating the objective function (see subsection \ref{subsec:framework:method}).  By employing the Bayesian optimization, we reduce the number of its evaluations, as it selects efficient sample evaluation points.  For additional applications of Bayesian optimization in inverse problems, readers may refer to \cite{Huang:2021:BOF, Vargas-Hernandes:2019:BOI, Hammerschmidt:2018:SIP}.

In addition to tracking the target, when the shape of the targets are unknown, we also propose a method to find the shape of the target via artificial neural networks. We parametrize the boundary of the target and estimate the parameters by learning a data-to-shape map. We apply a simple fully connected neural network (FCNN), following the approach in \cite{Gao:2022:ANN}. We remark that fine-tuned deep neural networks are recently introduced for solving inverse scattering problems \cite{Gao:2022:ANN, Guo:2022:PED, Chen:2020:ROD}.

Our devised approach tracks small and extended targets of general shapes with trajectories that may evolve randomly.  It is noteworthy that we use the training data only at the initial time to recover the target’s shape. If the shape of the target is known or well-identified, efficient tracking can be achieved using only far-field data of the moving target without further training data. We validate the proposed method and its performance through simulations using synthetic data with random trajectories, corrupted by random noise. 

The remainder of this paper is organized as follows. In section \ref{sec:ff}, we present the formulation of the tracking problem. In section \ref{sec:theory}, we establish the translation formula and the rotation formula for far-field data, which play the key roles in our work, and we verify the local Lipschitz property for the far-field data. Section \ref{sec:bo} describes a Bayesian optimization scheme for detecting the rotation angle of a target, and section \ref{sec:framework} outlines the proposed tracking algorithm for the moving target. Numerical experiments are conducted in section \ref{sec:num}. Section \ref{sec:fin} concludes the paper with a brief discussion.

\section{Problem formulation} \label{sec:ff}
We consider the scattering problem in two dimensions with a sound-soft scatterer $\Omega$ with a fixed wavenumber $k$:
\begin{align} 
\begin{cases} \label{prob:scattering}
\Delta u +k^2 u=0, & \text{in } \ \R^2\backslash \overline{\Omega},\\
u=0& \text{on } \ \p \Omega,\\
\lim_{r\to\infty} \sqrt{r}\left(\dfrac{\p u^s}{\p r}-iku^s\right)=0 & \mbox{with }r=|x|,
\end{cases}
\end{align}
where $u=u^i+u^s$ is the total field, $u^i(x)=e^{ikx\cdot d}$ the incident plane wave, and $u^s$ the scattered field. 
Our objective is to investigate the inverse problem of simultaneously tracking both the location and orientation of a moving target $\Om$ from far-field measurement data. We denote the target as $\Omega_t$ with time variable $t\geq 0$. Assuming the far-field pattern is measured at discrete time points for only a single incident field, either in full or limited view, our goal is to recover the target $\Omega_t$ at each discrete time point $t$.
In the remainder of this section, we introduce the definition and layer potential technique for the far-field pattern.

We assume that $\Omega$ is a simply-connected, bounded domain with a $C^2$ boundary. 
The fundamental solution of the two-dimensional Helmholtz equation in \eqref{prob:scattering} is given by
\[
\Phi(x,y) :=\frac{i}{4}H_0^{(1)}(k|x-y|),
\]
where $H_0^{(1)}$ is the Hankel function of the first kind of order $0$. By Green's formula, the scattered field $u^s$ can be expressed as
\begin{align} \label{eqn:u^s:orig}
u^s(x)\,=\, \int_{\p\Omega} \bigg[\frac{\p \Phi(x,y)}{\p n(y)} - i\eta \Phi(x,y) \bigg] \phi(y)\,ds(y), \quad  x \in \R^2 \backslash \p \Omega,
\end{align} 
where $n(y)$ is the outward normal vector to $\p\Omega$ at $y$, $\phi$ is a density function on $\p\Om$, and $\eta$ is a fixed nonzero real parameter.

Let $C(\p\Om)$ denote the space of continuous functions on $\p\Om$. We define the single-layer potential $S_\Omega:C(\p \Omega)\to C(\p \Omega)$ by
\begin{align} 
\notag (S_\Omega\phi)(x) & = \int_{\p \Omega} \Phi(x,y)\phi(y)\,ds(y)  \\ 
& = \frac{i}{4}\int_{\p \Omega} H_0^{(1)}(k|x-y|)\phi(y) \, ds(y) \label{eqn:slo}
\end{align}
and the integral operator $K_\Omega:C(\p \Omega)\to C(\p \Omega)$ by
\begin{align} 
\notag (K_\Omega\phi)(x) &= \int_{\p \Omega} \frac{\p \Phi(x,y)}{\p n(y)}\phi(y)\,ds(y)\\
& = \frac{ik}{4}\int_{\p \Omega} H_1^{(1)}(k|x-y|)\frac{\langle x-y, n(x)\rangle}{|x-y|}\phi(y) \, ds(y). \label{eqn:dlo}
\end{align}
Applying the jump relations for the layer potential operators to \eqref{eqn:u^s:orig}, one can easily derive 
\begin{align} \label{eqn:u^s:lp}
u^s(x)\,=\, \left(\frac{1}{2}I +K_\Omega - i\eta S_\Omega \right)\phi, \quad x \in \p\Omega.
\end{align}
By the boundary condition in \eqref{prob:scattering}, this reduces to the integral equation:
\begin{align} \label{eqn:density}
\left(\frac{1}{2}I +K_\Omega - i\eta S_\Omega \right)\phi \,=\, -u^i \quad \text{ on }  \p\Omega,
\end{align}
where $u^i =e^{ikx\cdot d}$. We note that \eqref{eqn:density} has a unique solution $\phi \in C(\p \Omega)$ (see \cite[Theorem 3.11]{Colton:1998:IAE}). For the direct problem, the scattered field $u^s$ corresponding to a given incoming field $u^i$ can be obtained by first solving $\phi$ in \eqref{eqn:density} and then evaluating the right-hand side in \eqref{eqn:u^s:lp}. 

\smallskip
\begin{remark} \rm
The discussion above extends naturally to Sobolev spaces. It is well known that the layer potential operators $S:H^{-1/2}(\p \Omega)\to H^{1/2}(\p \Omega)$ and $K:H^{1/2}(\p \Omega)\to H^{1/2}(\p \Omega)$ are bounded. The integral equation \eqref{eqn:density} has a unique solution $\phi \in H^{1/2}(\p \Omega)$, and the solution $u$ to problem \eqref{prob:scattering} lies in $H_{\text{loc}}^1(\R^2 \backslash \overline{\Omega})$. For further details, readers may refer to \cite{Colton:1998:IAE}. 
	\end{remark}

The scattered field $u^s$ has an asymptotic expansion in terms of the far-field pattern $u^\infty$ as
\beq\label{def:ff}
u^{s}(x) \, = \, \frac{e^{ik|x|}}{\sqrt{|x|}}\bigg[ u^{\infty}(\hat{x}) +O\bigg(\frac{1}{|x|} \bigg)\bigg], \quad \ |x| \to \infty,
\eeq
uniformly in all directions $\hat{x} =x/|x|\in S^1$, where $S^1=\{y\in\RR^2\,|\,\|y\|=1\}$.
By taking $|x|\to \infty$ in \eqref{eqn:u^s:lp}, we have
\begin{align}
u^s(x)\,&=\, \frac{e^{i\pi/4}}{\sqrt{8\pi k}}\int_{\p\Omega} \bigg[\frac{\p}{\p n(y)}\frac{e^{ik|x-y|}}{\sqrt{|x-y|}} - i\eta \frac{e^{ik|x-y|}}{\sqrt{|x-y|}} \bigg] \phi(y)\,ds(y) \notag \\
& = \, \frac{e^{ik|x|}}{\sqrt{|x|}}\left( \frac{e^{i\pi/4}}{\sqrt{8\pi k}} \int_{\p\Omega} \bigg[\frac{\p e^{-ik\hat{x}\cdot y}}{\p n(y)} - i\eta e^{-ik\hat{x}\cdot y}  \bigg] \phi(y)\,ds(y) +O\bigg(\frac{1}{|x|} \bigg) \right) \notag
\end{align}
and, by \eqref{def:ff},
\begin{align}
u^{\infty}(\hat{x}) \, & = \, \frac{e^{i\pi/4}}{\sqrt{8\pi k}}\int_{\p\Omega} \bigg[\frac{\p e^{-ik\hat{x}\cdot y}}{\p n(y)} - i\eta e^{-ik\hat{x}\cdot y} \bigg] \phi(y)\,ds(y) \notag \\
&= \, \frac{e^{-i\pi/4}}{\sqrt{8\pi k}}\int_{\p\Omega} \big(k\hat{x} \cdot n(y) +\eta\big)e^{-ik\hat{x}\cdot y} \phi(y)\,ds(y). \label{eqn:ff:orig} 
\end{align}
We may denote the far-field pattern as $u_\Om^\infty(\hat{x};\hat{d})$ to explicitly indicate the corresponding scatterer $\Om$ and the incident field direction $d$. Similarly, we may denote the density function satisfying \eqref{eqn:density} as $\phi_d$.

\section{Analysis on the far-field pattern under rigid motions} \label{sec:theory}
In two-dimensional Euclidean space, rigid motions consist of translations and rotations. 
We define the translated scatterer as $\Omega+\tau := \{z+\tau \,|\, z \in \Omega \}$ for $\tau \in \R^2$ and
the rotated scatterer as $R_\theta\Om:=\{R_\theta z\,|\, z\in\Om\}$ for $\theta \in [0,2\pi)$, where $R_\theta z$ denotes the rotation of $z$ by the angle $\theta$ around the origin. 

We assume that the scatterer at time $t$, denoted as $\Om_t$, is a rigid motion of the scatterer at the initial time. Our objective is to identify the translation vector $\tau$ and rotation (or orientation) angle $\theta$ for each $t$ satisfying 
\beq\label{Omt:rigid}\Om_t\approx R_\theta \Om_0 +\tau,
\eeq
under the assumption that the initial scatterer $\Om_0$ contains the origin. For that purpose, we design an algorithm to find $\tau$ and $\theta$ such that 
\beq\label{Omt:ff}
u^\infty_{\Om_t}(\hat{x};d)\approx u^\infty_{R_\theta \Om_0 +\tau}(\hat{x};d)
\eeq
for a fixed $d$ and various values of $\hat{x}$. In this section, we derive key formulas for the far-field patterns with respect to translation and rotation. These formulas will later be used to develop algorithms to identify the location and rotation angle of the scatterer, where the rotation angle is relative to a reference point contained in the scatterer.

\begin{remark}\label{rmk:nonuniqueness}
We note that the rotation angle $\theta$ satisfying \eqref{Omt:rigid} may not be unique. For example, if $\Om_0$ is a disk centered at the origin, then it remains invariant under the rotation $R_{\theta}$ for any angle $\theta$.
Consequently, the far-field data of $\Om_t$ also lacks uniqueness in determining the rotation angle. 
\end{remark}
The non-uniqueness in $\theta$ satisfying \eqref{Omt:ff} presents a challenge in identifying the rotation angle. This issue arises particularly when the scatterer has certain symmetrical properties, such as a disk centered at the origin. We develop an algorithm designed to track the rotation angle under the assumption that the rotation angle between two consecutive time points is constrained within a limited interval range and the scatterer is not symmetric under rotations $R_\theta$ for small angles $\theta$.

\subsection{Translation and rotation formulas of the far-field pattern} \label{subsec:theory:rigid}
 
We analyze the changes in the far-field pattern resulting from translations and rotations of the scatterer $\Om$. 
First, one can easily derive the translation formula as follows (see, for instance, \cite[Section 5.1]{Colton:1998:IAE}). For the reader's convenience, we provide the proof using the layer potential technique (refer to \cite{Ammari:2014:GPT} for the conductivity problems). 
\begin{lemma}\label{thm:translation}
Fix $\tau\in\RR^2$. For the incident field $u^i(z)=e^{ikz\cdot d}$ with the direction $d\in S^1$, the far-field pattern of $\Omega+\tau$ is given by
\begin{align} \label{eqn:translation}
u^\infty_{\Omega+\tau}(\hat x;d) \, =\, e^{-ik\tau\cdot (\hat{x}-d)}u^\infty_{\Omega}(\hat x;d)\quad\mbox{for all directions }\hat{x}\in S^1. 
\end{align}
\end{lemma}

\begin{proof}
For $z \in\p \Omega$ and $\phi \in C(\p \Omega)$, let $z^\tau=z+\tau$ and define $\phi^\tau\in C(\p(\Omega+\tau))$ by $\phi^\tau(z^\tau)=\phi(z)$. Using \eqref{eqn:slo}, \eqref{eqn:dlo}, and the translational invariance of the inner product, we observe that
\[S_{\Omega+\tau}[\phi^\tau](z^\tau) = S_\Omega[\phi](z)  \quad \text{ and } \quad K_{\Omega+\tau}[\phi^\tau](z^\tau) = K_\Omega[\phi](z). \]
Thus, given that $\phi$ is a solution to equation \eqref{eqn:density} on $\p \Omega$ with the incident field of direction $d$, it holds that for $z\in\p\Om$,
\begin{align*}
&\left(\frac{1}{2}I +K_{\Omega+\tau} - i\eta S_{\Omega+\tau} \right)[\phi^\tau](z^\tau) 
=\left(\frac{1}{2}I +K_{\Omega} - i\eta S_{\Omega} \right)[\phi](z)=-u^i(z)= -e^{ikz^\tau\cdot d}\, e^{-ik\tau\cdot d}.
\end{align*}
This implies that $e^{ik\tau\cdot d}\phi^\tau$ is the solution to \eqref{eqn:density} on $\p(\Omega+\tau)$. 
By applying \eqref{eqn:ff:orig} for $\Om+\tau$ and setting $y=z^\tau=z+\tau$ with $z\in\p \Om$, we prove \eqref{eqn:translation}.
\end{proof}

Next, we also derive the rotation formula by the layer potential technique:
\begin{lemma}\label{thm:rotation}
Fix $\theta\in S^1$. 
For the incident field $u^i(z)=e^{ikz\cdot d}$ with the direction $d\in S^1$, the far-field pattern of $R_\theta\Om$ is given by
\begin{align} \label{eqn:rotation}
u^\infty_{R_\theta \Omega}(\hat x;d) \, =\, u^\infty_{\Omega}(R_{-\theta}\hat x;R_{-\theta}d). 
\end{align}
\end{lemma}

\begin{proof}
For $z \in \p\Omega$ and $\phi \in C(\p \Omega)$, let $z^\theta=R_\theta z$ and define $\phi^\theta$ the continuous function on $\p(R_\theta \Omega)$ given by $\phi^\theta(z^\theta)=\phi(z)$. From \eqref{eqn:slo}, \eqref{eqn:dlo} and the rotational invariance of the inner product, we observe that
\[S_{R_\theta \Omega}[\phi^\theta](z^\theta) = S_\Omega[\phi](z)  \quad \text{ and } \quad K_{R_\theta\Omega}[\phi^\theta](z^\theta) = K_\Omega[\phi](z). \]
 Thus, given that $\phi$ is a solution to equation \eqref{eqn:density} on $\p \Omega$ with the incident field of direction $R_{-\theta} d$, that is, $u^i(z)=e^{ikz\cdot R_{-\theta}d}$, it holds that for $z\in\p\Om$,
\begin{align*}
&\left(\frac{1}{2}I +K_{R_\theta \Omega} - i\eta S_{R_\theta \Omega} \right)[\phi^\theta](z^\theta) 
= \left(\frac{1}{2}I +K_{ \Omega} - i\eta S_{ \Omega} \right)[\phi](z)
= -u^i(z)=-e^{ikz^\theta\cdot d}.
\end{align*}
This implies that $\phi^\theta$ is a solution to the equation \eqref{eqn:density} on $\p(R_\theta \Omega)$ with the incident field of direction $d$. Hence, from \eqref{eqn:ff:orig} for $R_\theta \Omega$ with letting $y=x^\theta$, we obtain the rotation formula \eqref{eqn:rotation}.
\end{proof}
Lemma \ref{thm:translation} and Lemma \ref{thm:rotation} are crucial for recovering the location and orientation of the moving scatterer.

\subsection{Stability of far-field patterns in rotation} \label{subsec:theory:stab}

We demonstrate the Lipschitz stability for the far-field pattern with respect to the rotation angle. This is essential for understanding the convergence behavior of the Bayesian optimization approach to infer the orientation angle of the scatterer in Section \ref{subsec:inf:angle}   
	
\begin{theorem} \label{thm:angle_stability}
Let $d \in S^1$ denote the direction of the incident field. There exists a constant $C=C(\Om,k,d)$ 
such that
\begin{align} \label{ineq:angle_stability}
\left| u_{\Omega}^\infty(\hat{x};d) - u_{R_\theta \Omega}^\infty(\hat{x};d) \right|
\leq C |\theta|  \quad\mbox{for } |\theta|\ll 1.
\end{align}
\end{theorem}
	
\begin{proof}
By \eqref{eqn:ff:orig} and \eqref{eqn:rotation}, we estimate
\begin{align}
& \notag \left| u_{\Omega}^\infty(\hat{x};d) - u_{R_\theta \Omega}^\infty(\hat{x};d) \right| 
=\left| u_{\Omega}^\infty(\hat{x};d) - u_\Omega^\infty(R_{-\theta}\hat{x}, R_{-\theta}d) \right| \\\label{eqn:estm:ff} 
& \ =  \frac{1}{\sqrt{8\pi k}}\bigg|\int_{\p\Omega} \left[\big(k\hat{x} \cdot n(y) +\eta\big)e^{-ik\hat{x}\cdot y} \phi_d(y) - \big(kR_{-\theta}\hat{x} \cdot n(y) +\eta\big)e^{-ikR_{-\theta}\hat{x}\cdot y} \phi_{R_{-\theta}d}(y) \right] \,ds(y)\bigg|.
\end{align}

For all $\alpha, \beta \in \R$, we have
\begin{align}\label{ineq:estm:exp}
|e^{i\alpha} - e^{i\beta}| & \leq  |\cos \alpha - \cos \beta|  + |\sin \alpha - \sin \beta|\leq  2|\alpha - \beta|.
\end{align}
Let $x\in\p\Om$ be given by $x = \|x\|(\cos s, \sin s)$ for some $s\in\RR$ and $y=(y_1,y_2)\in \RR^2$, then 
\(R_{-\theta}x = \|x\|(\cos(s-\theta), \sin(s-\theta))\)
and, thus,
\begin{align}\label{ineq:estm:dot}
 |{x}\cdot y - R_{-\theta}{x}\cdot y| & \leq \|x\|\big(\left|y_1(\cos s - \cos(s-\theta))\right| +   \left|y_2(\sin s - \sin(s-\theta))\right|\big)\leq \|x\| \|y\|_1 |\theta|
\end{align}
where $\|y\|_1=|y_1|+|y_2|$.

By \eqref{ineq:estm:exp} and \eqref{ineq:estm:dot} with $y=d\in S^1$, we have
\begin{align}\label{ineq:estm:key} 
 |e^{ikx\cdot d} - e^{ikx\cdot R_{-\theta}d}|  \leq 2k\left|x\cdot d - R_{-\theta}x \cdot d\right|  
 &\leq  2\sqrt{2}k\|x\| \|d\|_1 |\theta| \leq \eps(\theta),
\end{align}
where \(\eps(\theta) := 2\sqrt{2}kM_\Om |\theta|\) with \(M_\Omega = \max\{\|z\| \,:\, z \in \p \Omega \}\). 
The density functions $\phi_{d}$ and $\phi_{R_{-\theta}d}$ are solutions to \eqref{eqn:density} for $u^i(x) = e^{ikx\cdot d}$ and $u^i(x) = e^{ikx\cdot R_{-\theta}d}$, respectively. 
By the boundedness properties of the inverse of the operator $(\frac{1}{2}I+K_\Om-i\eta S_\Om)^{-1}:C(\p\Om)\rightarrow C(\p\Om)$  (see p.57 and Section 3.5 in \cite{Colton:1998:IAE}) and \eqref{ineq:estm:key}, there exists a constant $C$ depending on $\Om$ such that
\begin{align}\label{ineq:estm:density}
\| \phi_d - \phi_{R_{-\theta}d} \|_{L^\infty(\p\Om)} \leq C \eps(\theta)\quad\mbox{for }|\theta|\ll 1.
\end{align}

From \eqref{eqn:estm:ff} and \eqref{ineq:estm:density}, when $|\theta|$ is closed to $0$, we have
\begin{gather}\label{ineq:estm:ff}
\left| u_{\Omega}^\infty(\hat{x};d) - u_{R_\theta \Omega}^\infty(\hat{x};d) \right| \leq 
\frac{1}{\sqrt{8\pi k}}\,(I_1 +I_2)
\end{gather}
with 
\begin{align}\notag
I_1 &= \int_{\p\Omega} \Big|\big(k\hat{x} \cdot n(y) +\eta\big)e^{-ik\hat{x}\cdot y} - \big(kR_{-\theta}\hat{x} \cdot n(y) +\eta\big)e^{-ikR_{-\theta}\hat{x}\cdot y} \Big|\cdot |\phi_d(y)| \,ds(y),\\\label{ineq:estm:I2}
I_2 &=\int_{\p\Omega}\big|kR_{-\theta}\hat{x} \cdot n(y) +\eta\big| \cdot\| \phi_d - \phi_{R_{-\theta}d} \|_{L^\infty(\p\Om)} \, ds(y)\leq (k+\eta)\, C|\p \Omega| \eps(\theta),
\end{align}
where $|\p\Om|$ is the length of $\p\Om$. 
To estimate $I_1$, we find
\begin{align*}
&\Big|\big(k\hat{x} \cdot n(y) +\eta\big)e^{-ik\hat{x}\cdot y} - \big(kR_{-\theta}\hat{x} \cdot n(y) +\eta\big)e^{-ikR_{-\theta}\hat{x}\cdot y} \Big| \\
 \leq\, & \big|k(e^{-ik\hat{x}\cdot y} \hat{x} - e^{-ikR_{-\theta}\hat{x}\cdot y}R_{-\theta} \hat{x}) \cdot n(y)\big| + \big|\eta(e^{-ik\hat{x}\cdot y} - e^{-ikR_{-\theta}\hat{x}\cdot y})\big|\\
\leq \, &\left(\sqrt{2}k(2k\|y\|_1 +1)|\theta| + O(\theta^2) \right) +  2k\eta\|y\|_1 |\theta|,
\end{align*}
where each term in the last line follows from Lemma \ref{lem:angle_stability}, which will be discussed below, and \eqref{ineq:estm:key}.
Since $\phi_d\in C(\p\Om)$, we have $\|\phi_d\|_{L^\infty(\p\Om)} \leq M_d$ for some $M_d>0$. It follows that
\begin{align}\label{ineq:estm:I1}
\notag I_1 & \leq \int_{\p \Omega} \Big(2k\left((\sqrt{2}k+\eta)\|y\|_1+1) |\theta| + O(\theta^2)\right)\Big)M_d  \, ds(y) \leq C|\theta| + O(\theta^2)
\end{align}
for some constant $C$ depending on $\Omega$, $k$, and $d$.
Hence, by \eqref{ineq:estm:ff} and \eqref{ineq:estm:I2}, we prove the theorem. 
\end{proof}

\begin{lemma}\label{lem:angle_stability}
We have
\begin{align*}
\big\|e^{-ik\hat{x}\cdot y} \hat{x} - e^{-ikR_{-\theta}\hat{x}\cdot y}R_{-\theta} \hat{x} \big\|_2 \leq \sqrt{2}\left(2k\|y\|_1 +1\right)|\theta| + O(\theta^2)\quad\mbox{for }|\theta|\ll 1.
\end{align*}
\end{lemma}
	
\begin{proof}
Set $\hat{x} = (\cos s, \sin s)$ for some $s \in \R$. We have 
\beq\label{eqn:estm:polar}
\begin{aligned}
\big\|e^{-ik\hat{x}\cdot y} \hat{x} - e^{-ikR_{-\theta}\hat{x}\cdot y}R_{-\theta} \hat{x} \big\|_2^2
=&\ \left(e^{-ik\hat{x}\cdot y}\cos s - e^{-ikR_{-\theta}\hat{x} \cdot y}\cos(s-\theta)\right)^2\\
 &+ \left(e^{-ik\hat{x}\cdot y}\sin s- e^{-ikR_{-\theta}\hat{x} \cdot y}\sin(s-\theta)\right)^2.
\end{aligned}
\eeq
Recall that $\cos\theta = 1+O(\theta^2)$ and $\sin\theta=\theta+O(\theta^3)$. 
Using trigonometric identities, one can easily find that
$\cos(s-\theta)  = \cos s + \theta\sin s + O(\theta^2)$. 
Thus, by \eqref{ineq:estm:key}, we obtain
\begin{align*}
\left|e^{-ik\hat{x}\cdot y}\cos s - e^{-ikR_{-\theta}\hat{x} \cdot y}\cos(s-\theta)\right|
 \leq \,&\left|e^{-ik\hat{x}\cdot y} - e^{-ikR_{-\theta}\hat{x} \cdot y}\right||\cos s| + |\theta\sin s|+ O(\theta^2)\\
 \leq \,& \left(2k\|y\|_1+1\right) |\theta|+ O(\theta^2)
\end{align*}
and, similarly, 
$$\left|e^{-ik\hat{x}\cdot y}\sin\phi - e^{-ikR_{-\theta}\hat{x} \cdot y}\sin(\phi-\theta)\right| \leq  \left(2k\|y\|_1+1\right) |\theta| + O(\theta^2).$$ 
Therefore, by \eqref{eqn:estm:polar}, we prove the lemma. 
\end{proof}

	\section{Bayesian optimization for orientation angle} \label{sec:bo}
	Now suppose that $\Omega$ moves to $\Omega^*$ by a rigid motion on $\R^2$. Then we can write 
	\begin{align*}
	\Omega^* = R_{\theta^*}\Omega +\tau^*
	\end{align*}
	for some displacement vector $\tau^* \in \R^2$ and orientation angle $\theta^* \in [0,2\pi)$. 
	In this setting, we aim to find $\tau$ and $\theta$ satisfying
	\begin{align}\label{eqn:bo:ff} 
		u^{\infty}_{R_{\theta}\Omega +\tau} = u^\infty_{\Omega^*}\quad\mbox{for }\Omega^* = R_{\theta^*}\Omega +\tau^*.
	\end{align}
	In this section, we describe the method for determining $\theta$ by using Bayesian optimization. 	
		
	\subsection{Bayesian optimization}
	We begin with a brief introduction to Bayesian optimization.
    Consider a parameter space  $X$, and let $f:X \to \R$ be an unknown objective function that we want to minimize. (If we want to maximize $f$, we can instead minimize $-f$.) While various methods exist for optimizing $f$, many become impractical when evaluating $f$ is computationally expensive.
    In such cases, Bayesian optimization offers a promising alternative, as it efficiently evaluates the objective function, significantly reducing computational cost. 
	
	The Bayesian optimization yields a posterior optimizer by taking a prior over the objective function $f$ and utilizing acquisition functions that select effective sample points based on posterior updates from previous observations \cite{Brochu:2010:TBO}.  
	A Gaussian process (GP) is frequently chosen as the prior due to its flexibility: it allows closed-form computation of marginals and conditionals, making it suitable for a wide range of optimization problems \cite{O'Hagan:1978:CFO, Mockus:1994:ABA}. Several acquisition functions, such as Probability of Improvement (PI) \cite{Kushner:1964:NML}, Expected Improvement (EI) \cite{Mockus:1978:TGO, Jones:1998:EGO, Lizotte:2008:PBO}, and Upper Confidence Bound (UCB) \cite{Cox:1992:SMG, Cox:1997:SDO}, can be used to guide the selection of sample points, depending on the desired trade-offs in predicting optimal values with minimal evaluations. Notably, these three functions admit closed-form expressions under a GP prior. Leveraging the GP prior, Srinivas proposed the GP-UCB, a more systematic variant of UCB that further enhances performance \cite{Srinivas:2009:GPO}.
	
	To apply Bayesian optimization to the function $f$, we impose a prior distribution $P(f)$ for $f$. Let $x_i \in X$ represent the $i$-th sample point, and let $d_i$ be the noisy observation of $f$ at $x_i$ with the measurement error being modeled as the Gaussian noise. Then we can write
	\begin{align} \label{eqn:bo:observation}
	d_i = f(x_i)+\eps_i, \quad \text{ with } \eps_i \sim \mathcal{N}(0,\nu I)
	\end{align}
	where $\nu$ is the variance of the Gaussian noise so that $d_i \sim \mathcal{N}(f(x_i),\nu I)$. Given a set of observations $d_{1:n} = \{d_1,\cdots, d_n \}$,  the Bayesian posterior distribution of $f$, conditioned on the data, satisfies
	\begin{align} \label{rel:posterior}
	 P(f|d_{1:n}) & \propto P(d_{1:n}|f) P(f),
	\end{align}
	where $P(d_{1:n}|f)$ equals to the multivariate normal distribution of mean $(f(x_1),\cdots, f(d_n))$ and covariance matrix $\nu I_n$. With sufficient data points $d_{1:N}$ for some $N$, the optimizer $x^*$ of $f$ is obtained by maximizing the posterior distribution, i.e., by finding $x^*$ that maximizes $P(d_{1:N}|f) P(f)$.
	
	To efficiently construct a set of sample data $d_{1:N}$, we utilize an acquisition function $u$. After initially evaluating the function at some random points $x_1, \cdots, x_{i_0}$, the next sample point is iteratively selected by maximizing the acquisition function $u$ based on previous observations. In specific, after having $d_{1:i}$ with $i\geq i_0$, we take the next point
	\begin{align}  \label{eqn:bo:sample}
		x_{i+1} = \arg\max_{x\in X} u(x| d_{1:i})
	\end{align}
	where $u(x| d_{1:i})$ is conditioned on the previous data $d_{1:i}$. The objective function $f$ is then evaluated at $x_{i+1}$ to obtain $d_{i+1}$ by \eqref{eqn:bo:observation}.  This sampling process is repeated until the full set of data $d_{1:N}$ is collected. We remark that it is important to select an appropriate acquisition function and determine the optimal number of data points $N$ to balance the cost of evaluations while attaining high performance in optimization.

\subsection{Inference of the orientation angle}\label{subsec:inf:angle}
We now discuss the inference of the orientation angle based on Bayesian optimization in the framework of moving targets. 
	Then, in view of \eqref{eqn:bo:ff}, we set the objective function for $\theta$ on $[0,2\pi)$ by
	\begin{align} \label{eqn:bo:objective}
	f(\theta) = \min_{\tau\in \R^2} \|u^\infty_{R_\theta \Omega+\tau}-u_{\Omega^*}^\infty \|
	\end{align}  
	so that $f$ is minimized at $\tau=\tau^*$ and $\theta= \theta^*$.
	We assume $|\tau^*|$ and $|\theta^*|$ are not too large so that the inference process is feasible. For instance, if we assume $\|\tau^*\|<M$ for some $M>0$, we have the objective function
	\begin{align} 
		f(\theta) = \min_{\tau\in \R^2,\, \|\tau\| < M} \|u^\infty_{R_\theta \Omega+\tau}-u_{\Omega^*}^\infty \|.
	\end{align} 
	In practice, one can take time intervals small enough to achieve the assumption; this step is encapsulated in subsection \ref{subsec:framework:method}.

	 For each $\theta$, we can compute $f$ by using Adam optimizer \cite{Kingma:2014:AMS} for a two-layer neural network that examines $\tau$ around the origin. That is, the neural network minimizes the right-hand side of \eqref{eqn:bo:objective}. Due to the translation formula \eqref{eqn:translation}, this minimization is equivalent to finding approximate zeros of 
	\begin{align*}
		 e^{-ik\tau\cdot (\hat{x}-d)}u^\infty_{R_\theta\Omega}(\hat x) = u_{\Omega^*}^\infty
	\end{align*}
 	within the domain of $\tau$. For this optimization, it is sufficient to use a two-layer neural network by its approximating property that arises from the nonlinear hidden layer. Readers may refer to \cite{Cybenko:1989:ASS, Barron:1993:UAB, Barron:1994:AEB} for universal approximation theorems of two-layer neural networks and to \cite{Hornik:1994:DAR, Siegel:2020:ARN} for their developments with relaxed assumptions.

	Since evaluating the objective function $f$ is computationally expensive due to minimization, we apply Bayesian optimization to efficiently optimize $f$. Let $\theta_i$ be the $i$-th sample and $d_i$ be the noisy observation of $f$ at $\theta_i$ modeled as \eqref{eqn:bo:observation} so that
	\[
	d_i = f(\theta_i)+\eps_i, \quad \text{ with } \eps_i \sim \mathcal{N}(0,\nu I).
	\]
	where $\nu$ is the variance of the Gaussian noise. Given a prior distribution $P(f)$ and a set of observations $d_{1:n}$, the posterior distribution of $f$ is obtained by \eqref{rel:posterior}.
	We fix the number of data points $N$ and collect the data $d_{1:N}$ to determine $\theta^*$, the optimizer of $f$, that maximizes the right-hand side of \eqref{rel:posterior}.

For the prior distribution, $P(f)$, we impose Gaussian process priors on $\theta$. Specifically, we take $P(f(\theta)) \sim GP(m(\theta),k({\theta},{\theta}'))$ where $m$ is the mean function and $k$ is the covariance function. Given that the size of $|\theta|$ is assumed to be small, we here suppose $m$ to be the zero function. Moreover, by Theorem \ref{thm:angle_stability}, we anticipate that $f$ is locally convex at $\theta^*$, with at most linear proportionality. Consequently, we set $k$ to be Mat\'ern kernel of parameter $\nu = 0.5$, reflecting the likelihood that $f$ exhibits less smoothness near $\theta^*$. (Readers may recognize $k$ as the unsquared exponential kernel.)

For the acquisition function, $u$, we take the Expected Improvement (EI) function with $\xi=0.1$. This choice is to consider a balance between exploration and exploitation, as described in \cite{Lizotte:2008:PBO}, with introducing a slight bias toward exploration to efficiently find the optimizer. The EI function is often an effortless choice in practice since it does not require additional tuning parameters. Using this acquisition function, we iteratively select new sample points 
\[
\theta_{i+1} = \arg\max_{\theta} u(\theta| d_{1:i}).
\] 
in accordance with \eqref{eqn:bo:sample}. To further reduce the search time of $\theta_{i+1}$, we implement the Sequential Domain Reduction Transformation technique (See \cite{Stander:2002:RSD}).
	
By executing the Bayesian optimization, we identify the optimizer $\theta^*$ of $f$, which aligns with the orientation angle of the target $\Omega^*$. Figure \ref{fig:obj_func} visually illustrates the Bayesian optimization process for small and extended targets. The optimization was successfully completed in $N=8$ steps in both targets. 
Note that the small target exhibits steeper slopes compared to  the extended target (see the proof of Theorem \ref{thm:angle_stability}).
	
\begin{figure}[h!]
    \centering
    \includegraphics[width=0.62\textwidth]{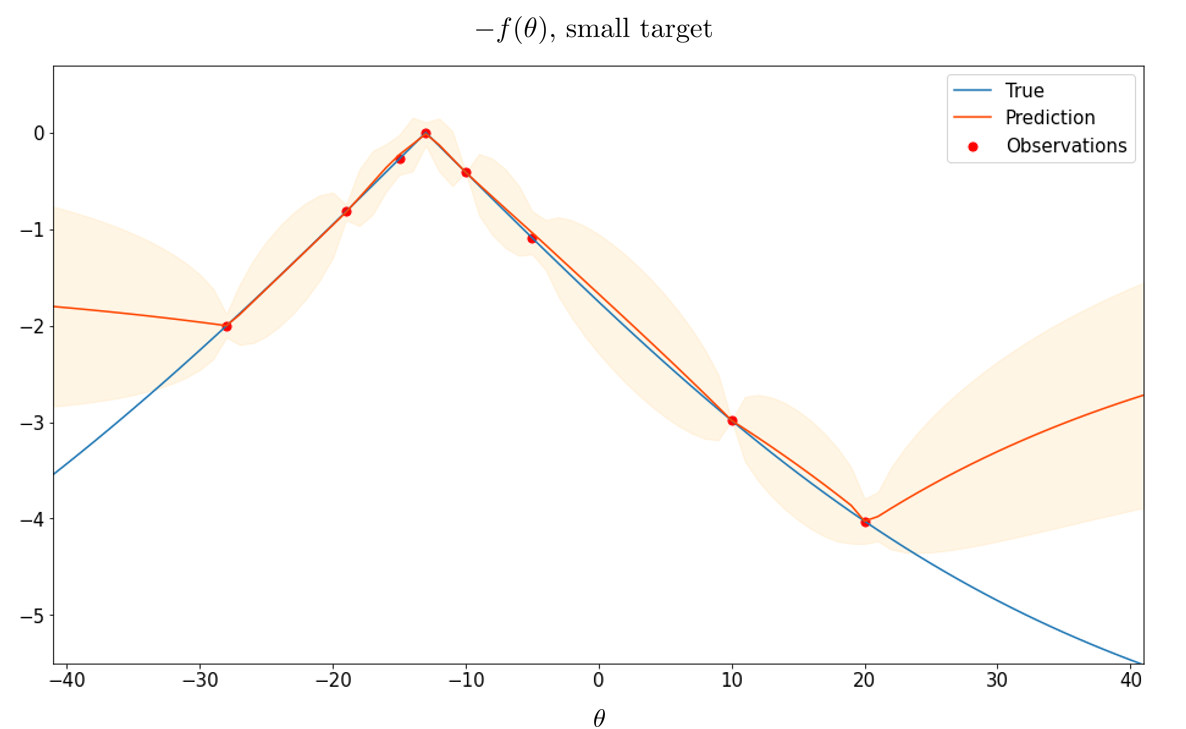}
    \vskip 3mm
    \includegraphics[width=0.62\textwidth]{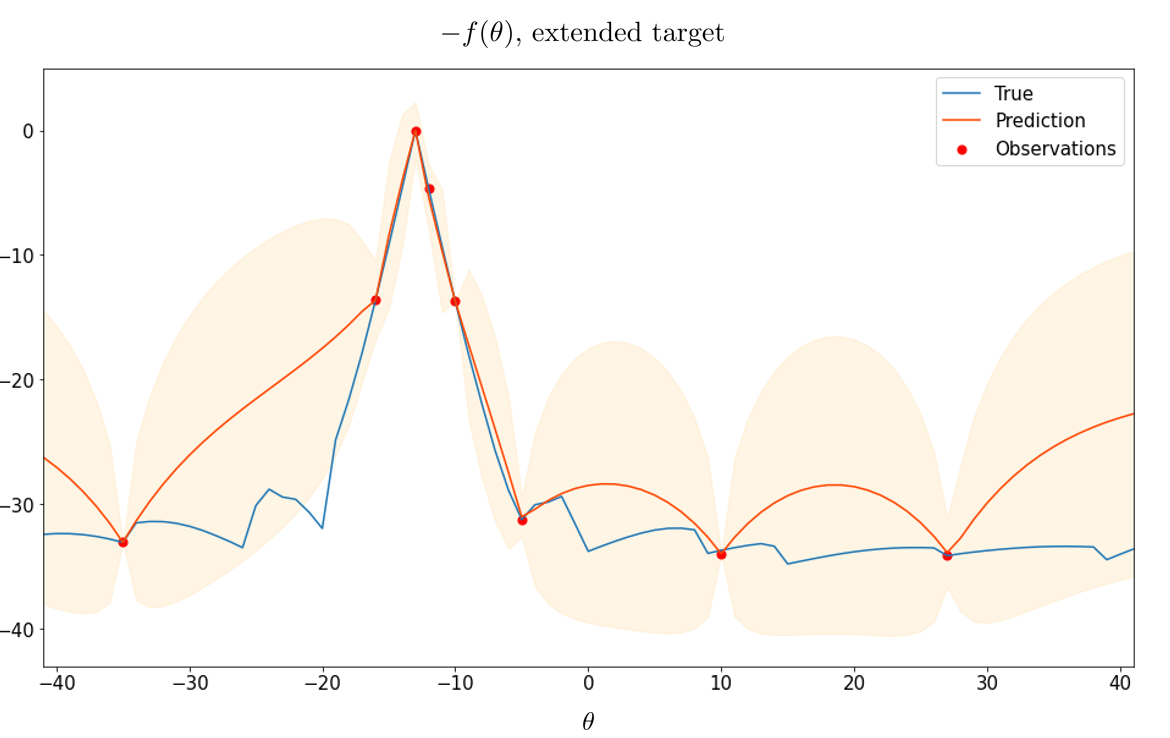}
    \caption{\label{fig:obj_func} 	
Plots of Bayesian optimization for the objective function $f$ in \eqref{eqn:bo:objective} for a small target (top) and an extended target (bottom). The red points indicate the examined sample points for $-f(\theta)$. The blue curves represent the true values, while the orange lines depict the predicted values during the optimization process. The orange shaded regions illustrate the confidence intervals corresponding to one standard deviation of the predictions.}
	\end{figure}

\section{Full framework for the tracking } \label{sec:framework}

	In this section, we present the complete framework for tracking a moving scatterer. The process begins with the scenario where the shape of the scatterer is unknown. We first reconstruct the boundary of the scatterer using a neural network and then proceed with the tracking scheme designed for scatterers with known shapes.

	\subsection{Deep learning for shape identification} \label{subsec:framework:shape}
	To identify the boundary of the scatterer, we use a fully connected neural network (FCNN) $\Phi$. An FCNN is a type of artificial neural network composed of a sequence of fully connected layers. Let $\Phi$ consist of $L+1$ layers ($L\geq 1$), with the $l$-th layer, $l=0,1,\cdots, L$ containing $n_l$ nodes. We then write $\Phi:\R^{n_0}\to \R^{n_L}$ as
	\begin{align} \label{eqn:nn}
	\Phi = \begin{cases}
	h_2 \circ \sigma \circ h_1, & L=1\\
	h_L \circ \sigma \circ h_{L-1} \circ \cdots \circ \sigma \circ h_1, & L \geq 2
	\end{cases}
	\end{align}  
	where $h_i: \R^{N_{i-1}} \to \R^{N_i}$ for $i=1,\cdots,L$ is defined by the affine transformation $h_i(x) = W_ix + b_i$. Here, $W_i \in \R^{N_i \times N_{i-1}}$ represents the weight matrix, $b_i \in \R^{N_i}$ is the bias vector, and $\sigma:\R \to \R$ is a nonlinear activation function that is applied element-wise.
    
    The term \textit{fully connected} indicates that no zero entries are pre-assigned in $W_i$, meaning every node in one layer is connected to all nodes in the next layer. To train the network $\Phi$, we optimize each weight $W_i$ and bias $b_i$ so that $\Phi$ approximates a desired mapping, denoted by $\Psi$. This is achieved by minimizing a loss function, which quantifies the difference between the outputs of $\Phi$ and $\Psi$, using gradient-based optimization and updating the parameters until the loss is sufficiently small.

	In our setting, the desired mapping $\Psi$ is the function that maps the far-field data to the parameters representing the boundary shape of the scatterer, i.e., $\p \Omega$. Therefore, $n_0$ represents the number of components of the far-field data, and $n_L$ corresponds to the number of boundary parameters. As Gao et al. discussed in \cite{Gao:2022:ANN}, with an adequate number of layers and sufficient training data, it is possible to reconstruct the boundary from the far-field data. This process requires careful control of the number of layers and data to efficiently manage resources.
	
	To improve the performance of the network, we discuss three important considerations in its design. First, we must account for the structure of the data. Since both the far-field data and the boundary parameters are complex-valued, it is inappropriate to use standard networks designed for real-valued parameters. Hence, we separate the parameters into either rectangular or polar form. The loss function must also be carefully selected to handle complex vectors.
	
	Next, we need to tune the weight parameters of the loss function. Since some parameters are more difficult to learn than others, it is necessary to apply a weighted loss function and adjust the weights to improve the output. For example, in the ellipse perturbation, the parameters representing the Fourier coefficients are relatively less sensitive in terms of performance, so we may assign higher weights to these parameters compared to the others (see subsection \ref{subsec:num:outline} or the Appendix).
	
	Lastly, we must choose an appropriate activation function to enhance the network's learning ability. Activation functions introduce nonlinearity to the model, increasing its representational capacity. However, some functions, including ReLU, may cause the vanishing gradient problem, which can hinder the learning process \cite{Lu:2020:DRI}. Therefore, it is important to prevent such issues by selecting a suitable activation function for the model. Readers may refer to \cite{Rasamoelina:2020:RAF} for a review of various activation functions.
	
	The architecture of the network in our work is a FCNN $\Phi:\R^{n_0}\to \R^{n_L}$ with $L=5$. For clarity, we denote $n_0$ and $n_L$ as $n_i$ and $n_o$, respectively, to explicitly represent the sizes of the input and output layers. Since the data are complex-valued, the network processes both the real and imaginary parts of the input, outputting their magnitudes and arguments. As a result, the sizes of $n_i$ and $n_o$ are twice the size of the input far-field data and twice the number of shape parameters, respectively. The number of nodes in each layer of $\Phi$ is $n_i$, $2n_i$, $3n_i$, $5n_o$, $2n_o$, and $n_o$, in the given order. 
	
    Additionally, we use SELU (Scaled ELU) activation functions to induce a self-normalizing effect in the network \cite{Klambauer:2017:SNN} and apply the LeCun uniform initialization to facilitate efficient learning. 
	The loss is computed using an $L^2$ weighted norm for complex vectors constructed from the output. The weights are set to $2.5$ for the Fourier coefficients parameters, and $1$ for the others. Figure \ref{fig:framework:dl} visualizes the network architecture, where $J$ denotes the number of shape parameters. 
	
	\begin{figure}[h!]
		\centering
		\includegraphics[width=0.95\textwidth]{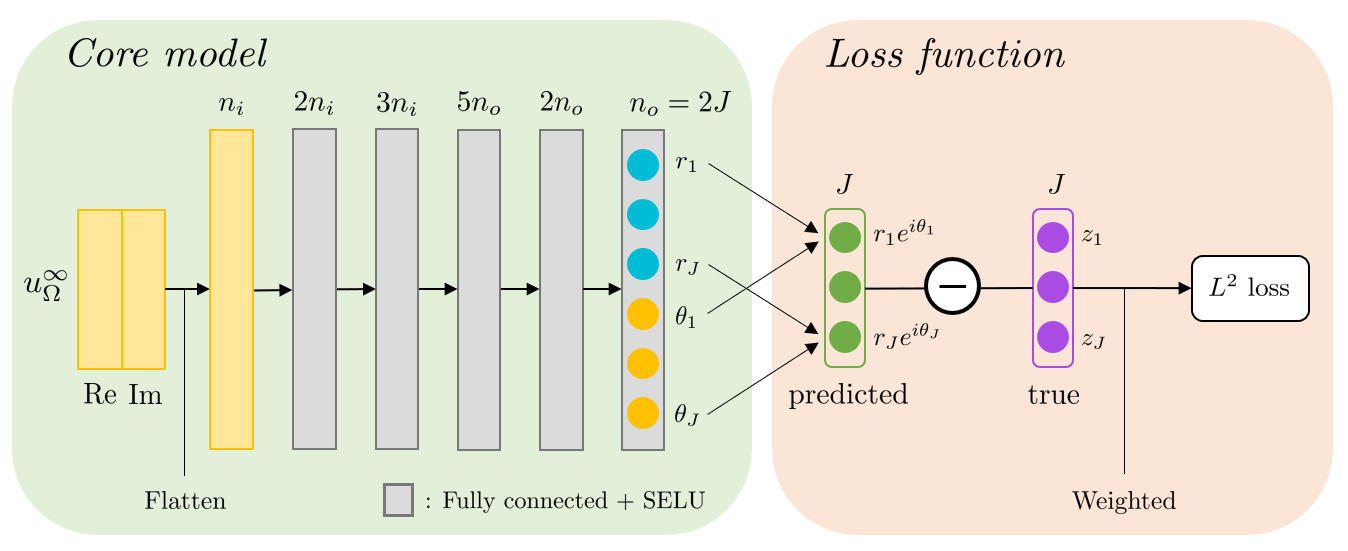}  
		\caption{\label{fig:framework:dl} Architecture of the neural network $\Phi$. The left side shows its core model, while the right side illustrates the loss function used for training. $J$  represents the number of shape parameters defining the boundary of the target $\p \Omega$.}
	\end{figure} 
	
	In this setting, we train the network using the Adam optimizer with a learning rate $0.0001$ and a mini-batch size of $128$ for each iteration. The network is trained over $5000$ epochs, with a total of $40000$ samples. Of these, $32000$ samples are used for the training set and $8000$ for the test set. Note that the sample data are generated by solving the direct scattering problem in \eqref{prob:scattering} using \eqref{eqn:density} and \eqref{eqn:ff:orig}.
	
    We emphasize that a pretrained network is desirable for practical use, and this can be accomplished by utilizing forward solvers. Therefore, in the framework presented below, we prepare the pretrained network using a high-quality dataset in advance.

	\subsection{Tracking procedure}  \label{subsec:framework:method}
	We now explain the method for tracking a moving scatterer using its far-field data. Recall that, due to rigid motions, once the shape of the object at the initial time is known, the problem reduces to estimating the displacement and rotation angle of $\Omega_t$ at each time $t>0$. By determining both the location and orientation of the target at each discretized time step, the entire tracking process can ultimately be achieved.

	Our method is based on the translation formula \eqref{thm:translation} and the Bayesian optimization of the objective function \eqref{eqn:bo:objective} to determine the location and orientation angle of the target scatterer. Assume that $\Omega_0 = \Omega$ contains the origin, and that $\Omega_t$ has an orientation angle $\theta_t$ and a displacement $\tau_t$ with respect to the origin, i.e.,
	\begin{align*}
	\Omega_t = R_{\theta_t}\Omega + \tau_t.
	\end{align*}
	This formulation is consistent with the consecutive movement of the target, since for any $t,s>0$, it holds that 
	\begin{align*} 
	R_{\theta_{t+s}}\Omega + \tau_{t+s} = R^{\tau_t}_{\theta_s}(R_{\theta_t}\Omega+\tau_t)+\tau_s
	\end{align*}
	where $R^{\tau}_\theta$ denotes the operator of a two-dimensional rotation by $\theta$ around $\tau\in \R^2$. The objective is to estimate the shape of $\Omega_0$ and to find $\tau_t$ and $\theta_t$ from the far-field data of $\Omega_t$, $t\geq 0$. 
	
	During the movement, we measure the far-field data at discrete times $t=n\delta$, where $\delta$ is the unit time step and $n\in\Z_{\geq 0}$. By discretizing, we denote $\Omega_n=\Omega_{n\delta}$ with  corresponding displacements $\tau_n=\tau_{n\delta}$ and orientation angles $\theta_n = \theta_{n\delta}$. We set the initial values as  $\tau_0=(0,0)$ and $\theta_0=0$, and assume $K$ receivers are equidistantly distributed. Then the measurement data using $M$ receivers are defined by
	\begin{align*}
	u^\infty_n := \begin{pmatrix}
	u^\infty_{\Omega_{n}}(\hat{x}_{i_0})\\
	\vdots\\
	u^\infty_{\Omega_{n}}(\hat{x}_{i_{M-1}})
	\end{pmatrix}
	\end{align*}
	where $\hat{x}_{i_m} = (\cos\theta_{i_m}, \sin\theta_{i_m})$ with $\theta_{i_m} = 2\pi i_m/K$ and $i_m \in \{0,1,\cdots, K-1\}$ for $m=0,1,\cdots, M-1$. Note that $M=K$ in the instance of full measurements.

	\paragraph{Steps in the tracking.}
	In this procedure, we choose $\delta$ to be sufficiently small so that $\tau_n$ and $\theta_n$ do not assume excessively large values. This helps improve the performance of the Bayesian optimization described in subsection \ref{sec:bo} by reducing the search domains for $\tau$ and $\theta$. We implement the Bayesian optimization by using the module from \cite{Nogueira:2014:BOO} with a total of $9$ steps. With this setup, the target can be tracked using the following procedure:
		When the shape of $\Omega_0$ is unknown, we append the following step 0 for the shape reconstruction.
		\begin{itemize}
			\item {\textbf{Step 0.}} Collect the far-field data of the object at $t=0$ and train the network described in section \ref{subsec:framework:shape} to learn the shape features of the object. Using the trained network, reconstruct the shape of $\Omega_0$ from the far-field data $u_{0}^\infty$.
		\end{itemize}
		
	\begin{itemize}
		\item {\textbf{Step 1.}} Given $\Omega_0$, generate the set of far-field data of $R_\theta \Omega_0$ over $\theta\in[0,2\pi)$. Note that the initial values are $(\tau_0, \theta_0)= ((0,0),0)$.
		\item {\textbf{Step 2.}} Given $(\tau_{n},\theta_{n})$,  perform Bayesian optimization around $\theta_n$ to estimate $\theta_{n+1}$ from $u^\infty_{n+1}$, i.e., minimize
		\begin{align} \label{eqn:objective:framework}
		f_n(\theta) = \min_{\tau\in B(\tau_n)} \|u^\infty_{R_{\theta}\Omega_0+\tau}-u^\infty_{n+1}\|
		\end{align}
		over $B(\theta_n)$ by Bayesian optimization, where $B(x)$ denotes a neighborhood of $x$. During the optimization, we enforce examination of $\theta_{n}$ as well as $\theta_n \pm \phi$ $(\in B(\theta_n) )$ for some $\phi>0$ to improve performance. 
		\item{\textbf{Step 3.}} Find the argument minimum $\tau_{n+1}$ of the right-hand side of \eqref{eqn:objective:framework} to complete the pair $(\tau_{n+1}, \theta_{n+1})$. Repeat Steps 2 and 3 iteratively until the tracking is complete.
	\end{itemize} 
	By using the shape constructed in Step 0 with the data $(\tau_n, \theta_n)$ obtained in Step 3, we can reconstruct the entire movement of the target scatterer. In specific, we estimate
	\begin{align} \label{eqn:recover}
	\Omega_n  = R_{\theta_n}\Omega_0 + \tau_n 
	\end{align}
	for each $n$.
	
	\begin{remark} \rm
		This algorithm is independent of both the shape and trajectory of the scatterer since it relies on Lemma \ref{thm:translation} and Theorem \ref{thm:angle_stability}. However, if the scatterer is large, the proof of Theorem \ref{thm:angle_stability} suggests that the slope around the true angle $\theta'$ could become steep, thereby increasing the error in the Bayesian optimization process in Step 2. This phenomenon is illustrated in Figure \ref{fig:obj_func}.
	\end{remark}

\section{Numerical simulations} \label{sec:num}
	We perform numerical experiments to evaluate the performance of our tracking algorithm. In the following subsection, we outline the experimental setup, and the subsequent subsection presents the results.
	
	\subsection{Modeling of a moving target} \label{subsec:num:outline}
	\paragraph{Shape of an extended target}	
We consider the boundary of each target as a perturbed ellipse. In terms of the complex variable $w$, one can express an ellipse as $e_0+\frac{e_1}{w}$ with some complex coefficients $e_0$ and $e_1$, where the normal vectors to the ellipse are direction $w-\frac{e_1}{w}$.
We assume the boundary curve of $\Omega$ is given by the image of $|w|=r$ via the mapping		
\begin{align} \label{eqn:bd}
		w \mapsto w+e_0+\frac{e_1}{w} +\eps\Big(w-\frac{e_1}{w} \Big) \cdot 2\text{Re}\bigg(\sum_{n=0}^{N-1} f_ne^{in \theta} \bigg),\quad w=r e^{i\theta}
		\end{align}
		for some $r>0$ and the Fourier coefficients $f_n$, $n=0,1,\dots, N-1$, of the associated perturbation function. Here, $\eps>0$ is a small parameter and $N$ is a positive integer. The parameters in this shape configuration are $r,\eps, e_0,e_1$, $N$, and $f_n$ (where $n=0,1,\cdots, N-1$). By allowing $e_1$ to vary, the perturbed ellipse provides greater flexibility in representing shapes compared to a disk perturbation (see, for instance, \cite{Choi:2021:ASR}).

		In the experiments, we consider extended targets with the wavenumber $k=1$. The boundaries of the scatterers are parameterized by \eqref{eqn:bd} and are required to have a diameter larger than $2\pi$. It is crucial to ensure that each boundary curve does not intersect itself and that the origin is enclosed by the curve. To generate target scatterers meeting these constraints, we derived specific conditions for the parameters in \eqref{eqn:bd}; see the appendix for details. Based on these conditions, we constructed a dataset consisting of $40,000$ samples, each characterized by $N=5$, $e_0=0$, $|e_1| \in [0,70]$, $|f_n| \in [0,1]$, $r \in [9.5,13.5]$ with $\eps = 0.01$. These settings enable representation of various shapes.
		
		For each target $\Omega$, we fix the direction of the incident field as $d = (1,0)$ and generate the far-field data \eqref{eqn:ff:orig} by solving the the integral equation \eqref{eqn:density}, with $\eta = k$ to ensure the stability of the equation \cite{Kress:1985:MCN}.
		Note that the far-field data obtained in this step correspond to $u^\infty_0$. In a similar manner, we generate the far-field data for $R_\theta \Omega$ over $\theta \in [0,2\pi)$, or for $\Omega_n$ that moves along its trajectory. All far-field data were computed using GYPSILAB on MATLAB \cite{Alouges:2018:FBS}.

\paragraph{Movement of a target.}
	We model the movement of a target scatterer $\Om_t$ over time by treating its location $\tau_t$ and angle $\theta_t$ as random processes. For the location $\tau_t$, we let the velocity of the target at $t$ be given by
		\begin{align*}
		v_t = v_0 + \sigma_v W^{(2)}_t
		\end{align*} 
		where $v_0\in\R^2$ is the initial velocity, $\sigma_v>0$ is the standard deviation, and $W^{(2)}_t$ is the standard two-dimensional Brownian motion. We denote $v_{n}= v_{n\delta}$ to be consistent with the previous notation. The location $\tau_t$ then satisfies
		\begin{align*}
		\tau_t = \int_0^t v_s ds
		\end{align*} 
		as we set $\tau_0=0$. Using stochastic theorems, we can express the dynamics as
		\begin{align} \label{traj:dynamics:v}
		\begin{pmatrix}
		v_{n+1} \\ \tau_{n+1} 
		\end{pmatrix} = 
		\mathbb{F} \begin{pmatrix}
		v_{n} \\ \tau_n 
		\end{pmatrix} + \mathbb{A}, \quad
		\mathbb{F} = \begin{pmatrix}
		\mathbb{I}_2 & 0 \\
		\tau \mathbb{I}_2 & \mathbb{I}_2 
		\end{pmatrix},
		\end{align}
		where $\mathbb{A}$ follows a four-variate normal distribution with mean zero and covariance matrix $\Sigma$ given by 
		\begin{align*}
		\Sigma = 
		\delta\begin{pmatrix}
		\sigma_v^2 \mathbb{I}_2 & \sigma_v^2/2\cdot \mathbb{I}_2 \\
		\sigma_v^2/2 \cdot \mathbb{I}_2 & \sigma_v^3/3\cdot \mathbb{I}_2 
		\end{pmatrix}.
		\end{align*}
Next, we let rotation angle $\theta_t$ of the target at $t$ be given by
		\begin{align*} 
		\theta_t = \theta_0 + \sigma_\theta W_t
		\end{align*}
		where $\theta_0\in[0,2\pi)$ is the initial orientation, $\sigma_\theta>0$ is the standard deviation, and $W_t$ is the standard one-dimensional Brownian motion. We then have
		\begin{align} \label{traj:theta}
		\theta_{n+1} = \theta_n + A_\theta,
		\end{align}
		where $A_\theta$ follows a normal distribution with mean zero and variance $\sigma_\theta^2$.
		In \cite{Ammari:2013:TMT}, the moving target models \eqref{traj:dynamics:v} and \eqref{traj:theta} are used to address the tracking problem for the planar conductivity equation with the generalized polarization tensors.

    In all simulation examples, we set the parameters $\sigma_v \in [1,2]$  and $\sigma_\theta = [1, 1.5]$ to simulate mild variations in the locations and orientations of the target. For each test target $\Omega$, sample trajectories were generated up to $n=80$.

    \subsection{Examples} \label{subsec:num:ex}
    We present the results of numerical simulations based on the setup described in the previous subsection. For the dataset of far-field data, we used $24$ equidistantly distributed receivers for the full view and added Gaussian white noise with levels of $5\%, 10\%$, and $20\%$. In all examples, Bayesian optimization was performed with a search radius of $\phi=50\degree$. (Refer to subsection \ref{subsec:framework:method}.)  Figure \ref{fig:full:shape} shows the shape identification results, while Figure \ref{fig:full:traj} demonstrates sample tracking results on the $xy$-plane, both serving as illustrating examples.

	\begin{figure}[h!]
		\centering
		\includegraphics[width=0.9\textwidth]{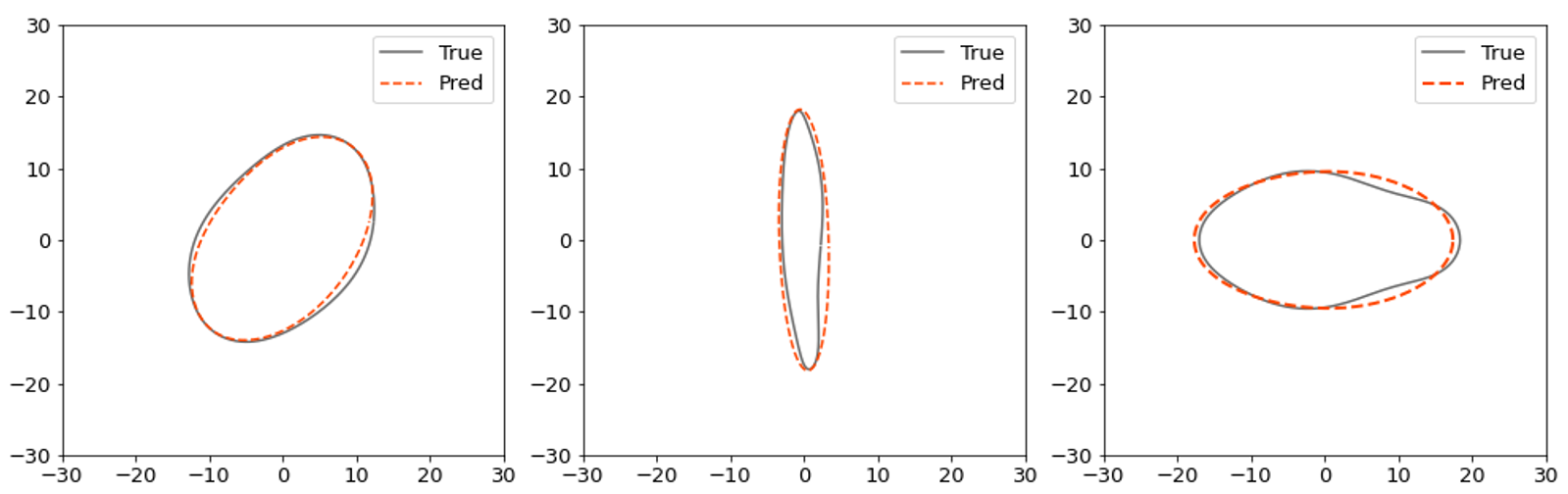}  
		\caption{\label{fig:full:shape} Plots of the shape identification for the boundaries of targets. The gray line represents the true boundary, while the orange dashed line indicates the predicted boundary.}
	\end{figure} 
	
	\begin{figure}[h!]
		\centering
		\includegraphics[width=0.9\textwidth]{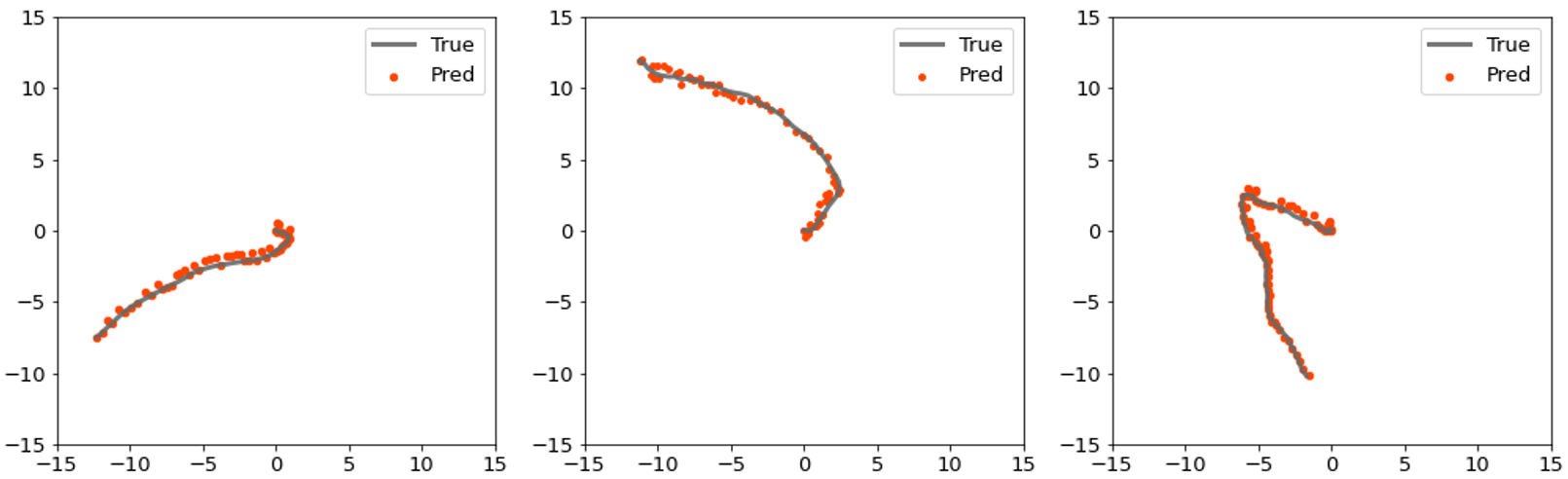} 
		\caption{\label{fig:full:traj} Plots of the tracking results on the $xy-$plane for far-field data with $5\%$ noise. The gray line represents the true trajectory, while the orange points show the predicted locations at each step.} 
	\end{figure} 

	Throughout the examples, one step of tracking the location and orientation took between 20 to 30 seconds. This time is reasonable, given the complexity of tracking angles and the accuracy presented in the following examples.
	
	\subsubsection{Tracking in full view with an unknown shape}
				
	Figure \ref{fig:full:result:unknown} shows the tracking results for an unknown shape under various noise levels in the full-view setting. The results demonstrate that the true trajectory and the predicted trajectory closely coincide. However, data with higher noise levels tend to exhibit larger deviations in the tracking results.
		
		\begin{figure}[h!]
			\centering
			\includegraphics[width=0.9\textwidth]{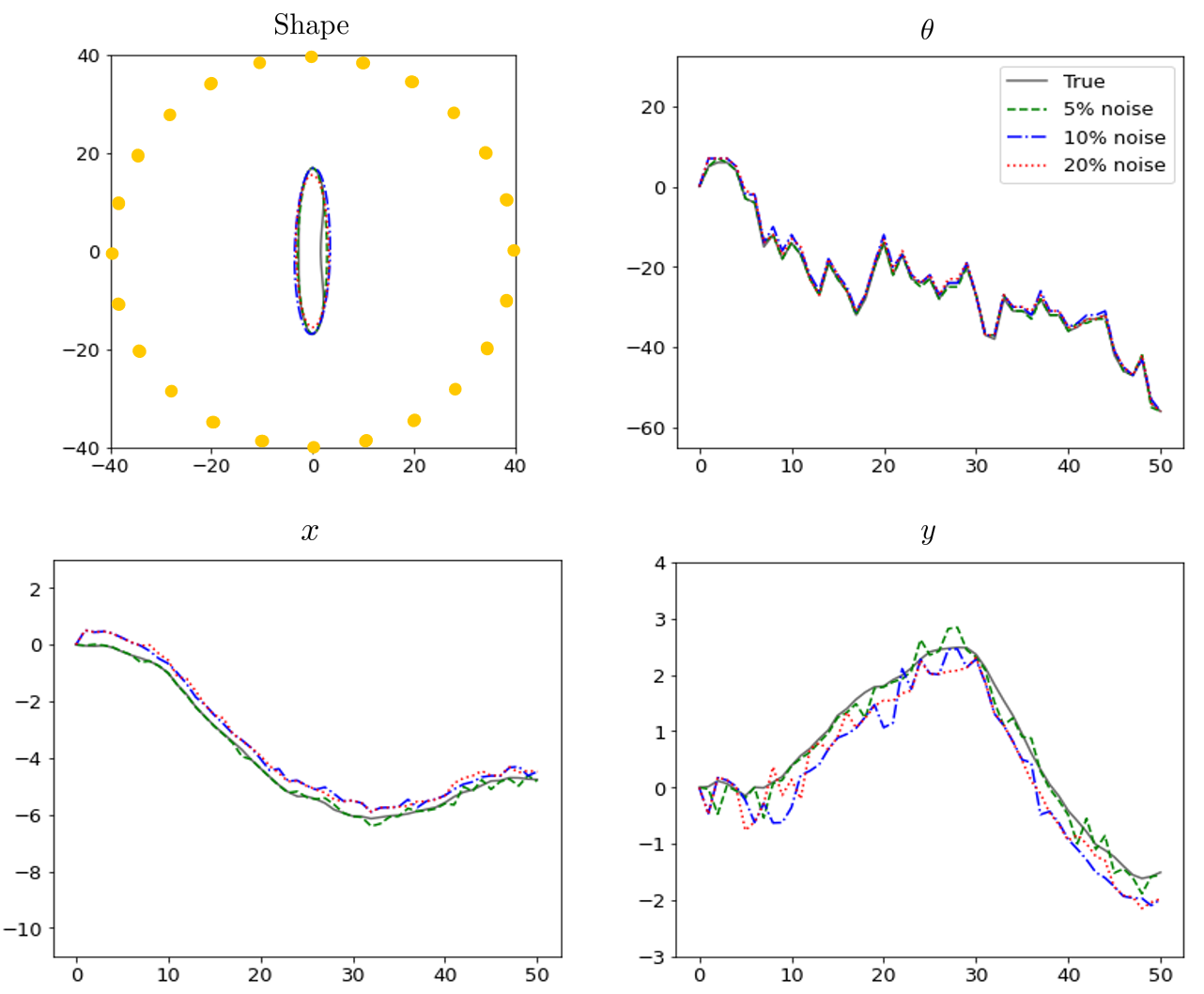} 
			\caption{\label{fig:full:result:unknown} Tracking results for an unknown shape with various noise levels in the full-view setting. The shape identification is shown in the top-left plot. Tracking results for the $x$-position (bottom-left), $y$-position (bottom-right), and orientation angle (top-right) are presented. The yellow marks indicate the measurement directions.}
		\end{figure}

\subsubsection{Tracking in limited view with a unknown shape}	
	 We now consider an unknown shape where the measurements are limited to half and a quarter of the full view. As the number of measurements decreases, the performance of shape identification deteriorates. Figures \ref{fig:half:result:unknown} and \ref{fig:quarter:result:unknown} present the tracking results under various noise levels for the half and quarter measurement configurations, respectively. These results demonstrate that as the noise increases, the tracking of the location deviates further from the true trajectory. Furthermore, the deviation in the tracking angle becomes more certain as the number of measurements decreases.

		\begin{figure}[h!]
			\centering
			\includegraphics[width=0.9\textwidth]{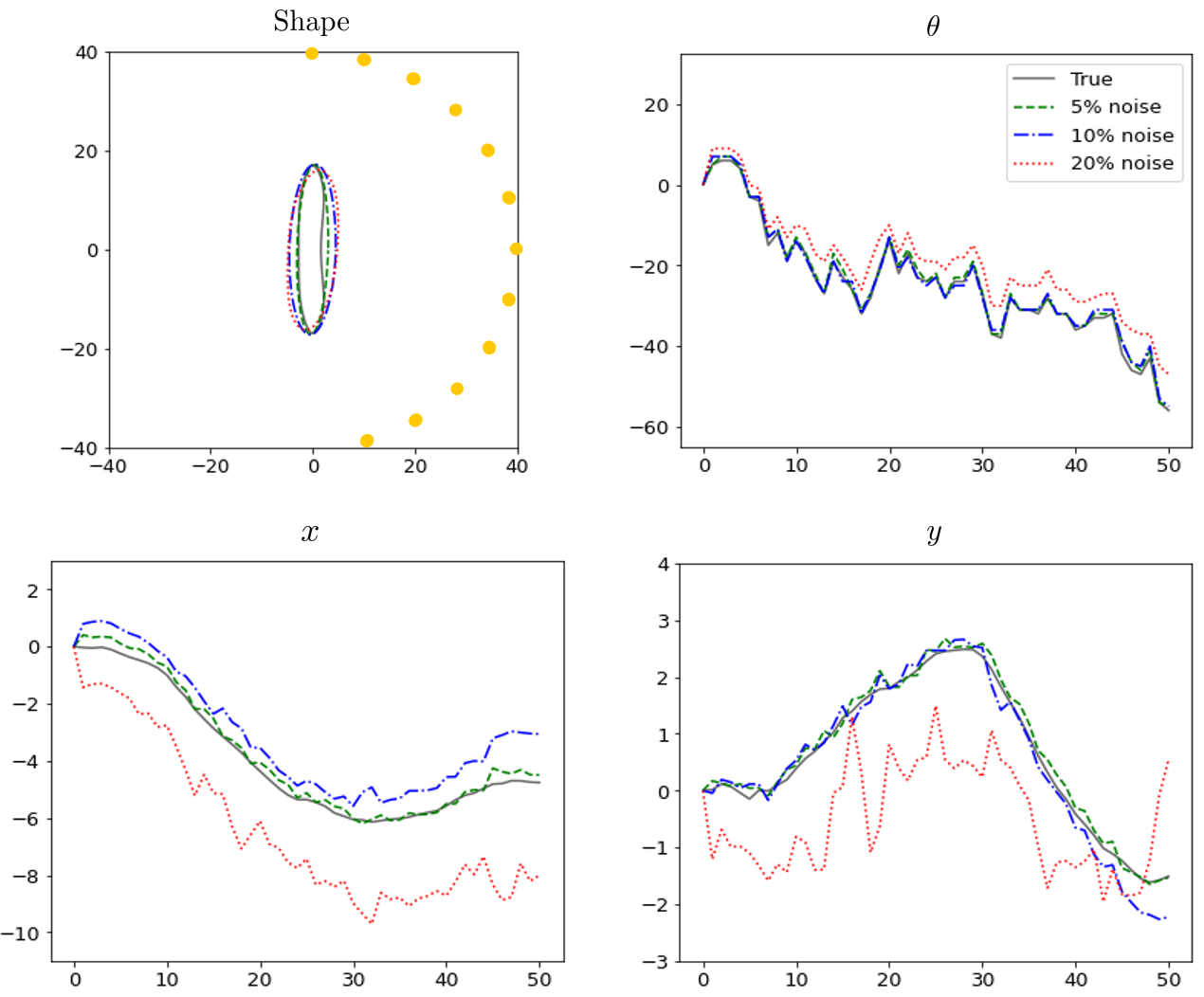}  
			\caption{\label{fig:half:result:unknown} Tracking results for an unknown shape with various noise levels in the half-view setting. The shape identification is shown in the top-left plot. Tracking results for the $x$-position (bottom-left), $y$-position (bottom-right), and orientation angle (top-right) are presented. The yellow marks indicate the measurement directions.}
		\end{figure}

		\begin{figure}[h!]
			\centering
			\includegraphics[width=0.9\textwidth]{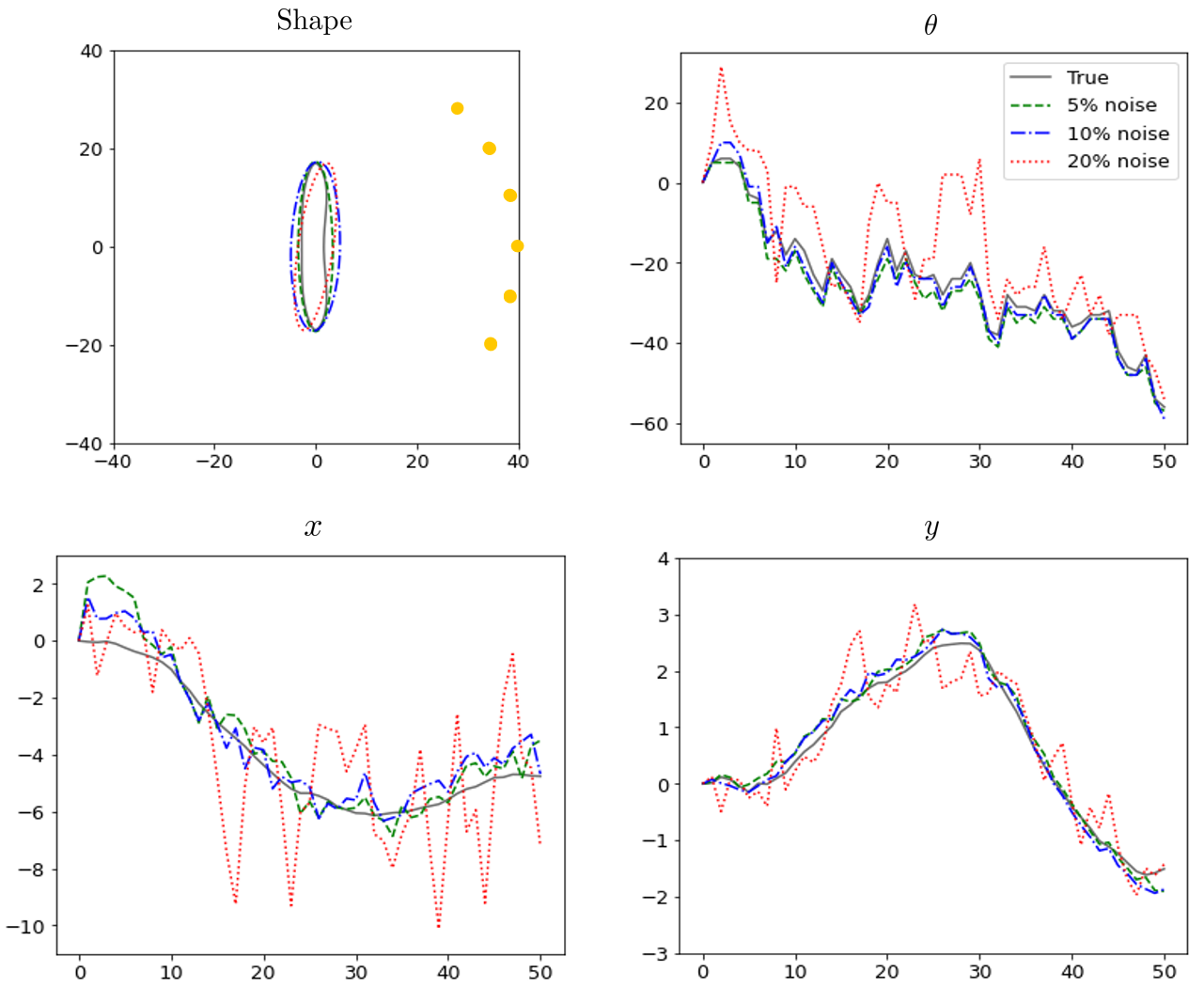}  
			\caption{\label{fig:quarter:result:unknown} Tracking results for an unknown shape with various noise levels in the quarter-view setting. The shape identification is shown in the top-left plot. Tracking results for the $x$-position (bottom-left), $y$-position (bottom-right), and orientation angle (top-right) are presented. The yellow marks indicate the measurement directions.}
		\end{figure} 

\subsubsection{Tracking in limited view with a known shape}	
	Finally, we consider the scenario where the measurements are limited to a quarter of the full view, but the shape of the target is known. In this case, the target shape is more intricate than in the previous simulations, yet both the location and angle tracking exhibit satisfactory performance. As observed consistently throughout the experiments, large deviations from the true values occur when the noise levels are higher.
		
		\begin{figure}[h!] 
			\centering
			\includegraphics[width=0.9\textwidth]{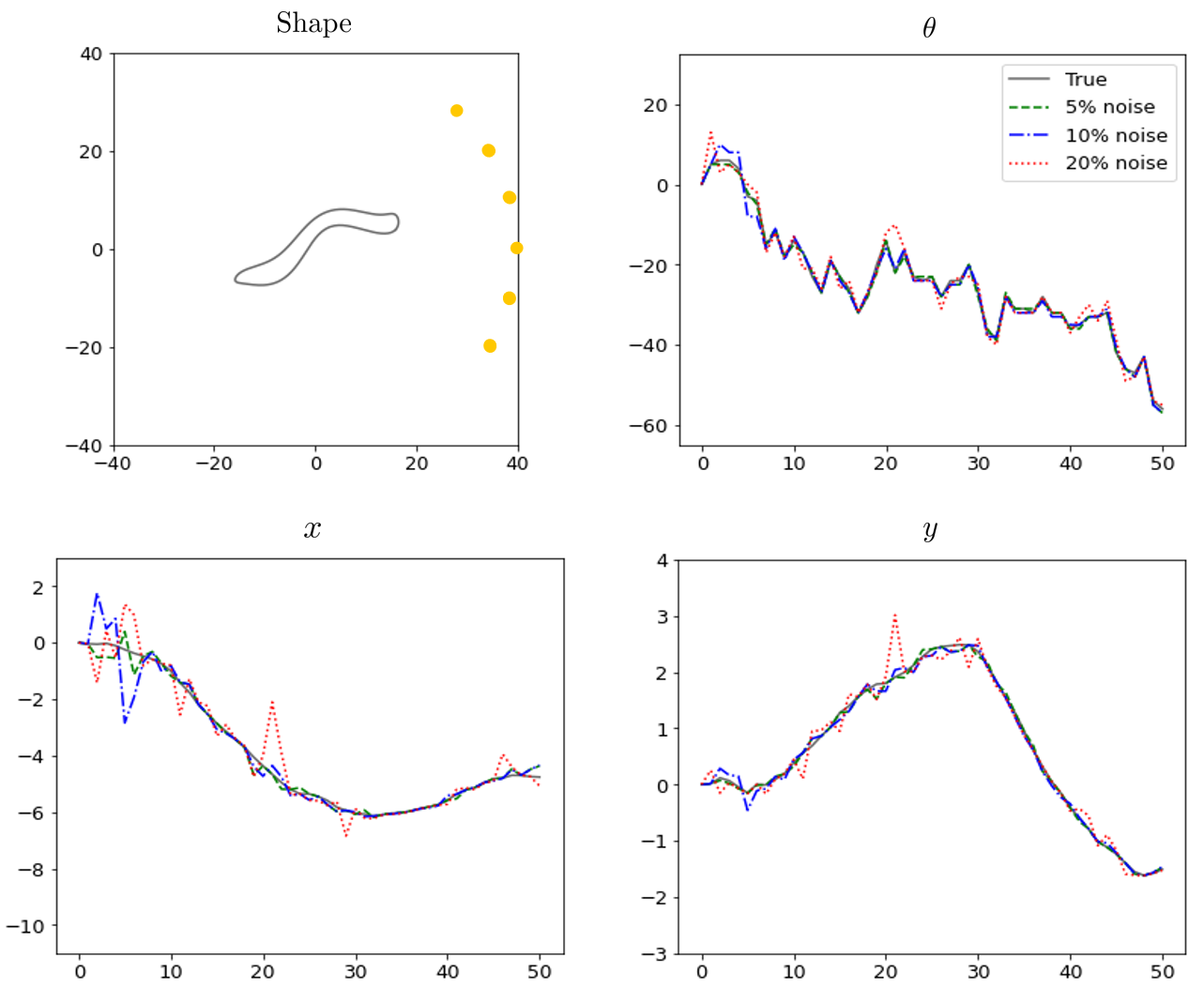}  
			\caption{\label{fig:limited:result:unknown} Tracking results for an unknown shape with various noise levels in the quarter-view setting. The shape identification is shown in the top-left plot. Tracking results for the $x$-position (bottom-left), $y$-position (bottom-right), and orientation angle (top-right) are presented. The yellow marks indicate the measurement directions.}
		\end{figure}

\section{Conclusion} \label{sec:fin}
We developed a tracking algorithm to determine the location and orientation of a moving scatterer from its far-field data, utilizing Bayesian optimization combined with the analytic properties of far-field data. To identify the shape of the target, we train a model using the data only from the initial step.
Numerical simulations under various measurement configurations demonstrated the robustness and effectiveness of the proposed method. Notably, when the target's shape is known or accurately reconstructed, the algorithm achieves high accuracy in tracking even with limited and noisy data. 
Our future work will focus on extending these results to three-dimensional problems and improving the optimization framework.

\begin{appendices}
\section{Generation of random shapes as perturbed ellipses}
Consider the ellipse perturbation in the form of \eqref{eqn:bd}. As we assume that the domain contains the origin, we simply set $e_0=0$. 
We discuss how to handle the shape generation based on the perturbed ellipse expression \eqref{eqn:bd}. We note that arbitrary choices for the parameters in \eqref{eqn:bd} may result intersecting curves or fail to enclose the origin in its interior.

	Recall that $w=re^{i\theta}=r\cos\theta + ir\sin\theta$. 
	 If we let $e_1=\xi+i\eta$, then the base ellipse $w+\frac{e_1}{w}$ has the parametrization:
	 	\begin{align*}
	\left[\left(r+\frac{\xi}{r} \right)\cos\theta + \frac{\eta}{r} \sin\theta \right] + \left[\left(r-\frac{\xi}{r} \right)\sin\theta + \frac{\eta}{r} \cos\theta\right]i, \qquad 0\leq \theta < 2\pi.
	\end{align*}
	For this ellipse, the squared distance of its boundary point from the origin is given by
	\begin{align*}
	r^2+\frac{\xi^2 + \eta^2}{r^2} + 2\xi\cos{2\theta} + 2\eta \sin{2\theta}.
	\end{align*}
Hence, the distance between the origin and the boundary of the ellipse has a lower bound
	\begin{align}\label{eqn:distance1}
	r-\frac{\sqrt{\xi^2+\eta^2}}{r}.
	\end{align}
	Likewise, we see that the perturbed term is bounded as follows:
	\begin{align}\label{eqn:distance2}
\eps\left|\Big(w-\frac{e_1}{w} \Big) \cdot 2\text{Re}\bigg(\sum_{n=0}^{N-1} f_ne^{in \theta} \bigg)\right|
\leq 2N\eps \|{f}_n\|_{\infty } \bigg(r+\frac{\sqrt{\xi^2+\eta^2}}{r}\bigg).
	\end{align}
For data generation, we impose a constraint on $f_n$ such that the lower bound in \eqref{eqn:distance1} is greater than the upper bound in \eqref{eqn:distance2}, i.e.,
	\begin{align} \label{ineq:cond}
	\|{f}_n\|_\infty < \frac{r^2-\sqrt{\xi^2+\eta^2}}{r^2+\sqrt{\xi^2+\eta^2}} \cdot \frac{1}{2N \eps}.
	\end{align}
This ensures that the origin is enclosed by the curve generated by \eqref{eqn:bd}. Specifically, by taking $\|f_n\|_\infty \leq 1$, and $\eps=0.01$, we have $2N\|{f}_n\|_{\infty }\eps \leq 0.1$. Therefore, by selecting a sufficiently large $r$ with an appropriate $e_1$ and $f_n$ so that \eqref{ineq:cond} is satisfied, we obtain a desired type of the boundary. 

\end{appendices}


\bibliographystyle{amsplain} 
\providecommand{\bysame}{\leavevmode\hbox to3em{\hrulefill}\thinspace}
\providecommand{\MR}{\relax\ifhmode\unskip\space\fi MR }

\providecommand{\MRhref}[2]{%
	\href{http://www.ams.org/mathscinet-getitem?mr=#1}{#2}
}
\providecommand{\href}[2]{#2}

\end{document}